\documentclass[12pt]{amsart}

  \usepackage{a4wide}
  \usepackage{amssymb}
  \usepackage{amsmath}
  \usepackage{amsthm}
  \usepackage{amsfonts}

  \def\mm{{\mathfrak m}}

  \def\NZQ{\mathbb}               

  \def\FF{{\NZQ F}}
  \def\GG{{\NZQ G}}

  \def\frk{\mathfrak}               

  \def\mm{{\frk m}}
  
  \def\nn{{\frk n}}
  \def\Phi{{\frk N}}
  \def\xb{{\mathbf x}}

  \def\opn#1#2{\def#1{\operatorname{#2}}} 
  \opn\chara{char} \opn\length{\ell} \opn\pd{pd} \opn\rk{rk}
  \opn\projdim{proj\,dim} \opn\injdim{inj\,dim} \opn\rank{rank}
  \opn\depth{depth} \opn\grade{grade} \opn\height{height}
  \opn\embdim{emb\,dim} \opn\codim{codim}
  
  \opn\Tr{Tr} \opn\bigrank{big\,rank}
  \opn\superheight{superheight}\opn\lcm{lcm}
  \opn\trdeg{tr\,deg}
  \opn\reg{reg} \opn\lreg{lreg} \opn\ini{in} \opn\lpd{lpd}
  \opn\size{size}\opn{\mult}{mult}
  \opn\div{div} \opn\Div{Div} \opn\cl{cl} \opn\Cl{Cl}
  \opn\Spec{Spec} \opn\Supp{Supp} \opn\supp{supp} \opn\Sing{Sing}
  \opn\Ass{Ass} \opn\Min{Min}
  \opn\Ann{Ann} \opn\Rad{Rad} \opn\Soc{Soc}
  \opn\Syz{Syz} \opn\Im{Im} \opn\Ker{Ker} \opn\Coker{Coker}
  \opn\Am{Am} \opn\Hom{Hom} \opn\Tor{Tor} \opn\Ext{Ext}
  \opn\End{End} \opn\Aut{Aut} \opn\id{id} \opn\ini{in}
  
  \def\project{{\kappa}}
  \opn\nat{nat}
  \opn\pff{pf}
  \opn\Pf{Pf} \opn\GL{GL} \opn\SL{SL} \opn\mod{mod} \opn\ord{ord}
  \opn\Gin{Gin}
  \opn\Hilb{Hilb}\opn\adeg{adeg}\opn\std{std}\opn\ip{infpt}
  \opn\Pol{Pol}
  \opn\sat{sat}
  \opn\Var{Var}
  \opn\Gen{Gen}
  \opn\homogen{homogen}
  
  \opn\aff{aff} \opn\con{conv} \opn\relint{relint} \opn\st{st}
  \opn\lk{lk} \opn\cn{cn} \opn\core{core} \opn\vol{vol}
  \opn\link{link} \opn\star{star}
  \opn\gr{gr}
  
  \def\Sc{{\mathcal S}}

  \def\pot#1#2{#1[\kern-0.28ex[#2]\kern-0.28ex]}

  \opn\dirlim{\underrightarrow{\lim}}
  \opn\inivlim{\underleftarrow{\lim}}
  \let\sect=\cap
  \let\dirsum=\oplus
  \let\tensor=\otimes
  \let\iso=\cong

  \let\Dirsum=\bigoplus

  \let\to=\rightarrow
  \let\To=\longrightarrow
  \def\Implies{\ifmmode\Longrightarrow \else
        \unskip${}\Longrightarrow{}$\ignorespaces\fi}
  \def\implies{\ifmmode\Rightarrow \else
        \unskip${}\Rightarrow{}$\ignorespaces\fi}
  \def\iff{\ifmmode\Longleftrightarrow \else
        \unskip${}\Longleftrightarrow{}$\ignorespaces\fi}

  \let\:=\colon
  \makeatletter
  \CheckCommand*\refstepcounter[1]{\stepcounter{#1}%
      \protected@edef\@currentlabel
       {\csname p@#1\endcsname\csname the#1\endcsname}%
  }
  \renewcommand*\refstepcounter[1]{\stepcounter{#1}%
    \protected@edef\@currentlabel
      {\csname p@#1\expandafter\endcsname\csname the#1\endcsname}%
  }
  \def\labelformat#1{\expandafter\def\csname p@#1\endcsname##1}
  \DeclareRobustCommand\Ref[1]{\protected@edef\@tempa{\ref{#1}}%
     \expandafter\MakeUppercase\@tempa
  }
  \makeatother

  \makeatletter
  \newcommand{\numberlike}[2]{%
     \expandafter\def\csname c@#1\endcsname{%
         \expandafter\csname c@#2\endcsname}%
  }
  \makeatother

 \def\DefaultNumberTheoremWithin{section}

 \theoremstyle{plain}
  \newtheorem{Lemma}{Lemma}
     \numberwithin{Lemma}{\DefaultNumberTheoremWithin}
     \labelformat{Lemma}{Lemma~#1}
  
     \numberwithin{Claim}{\DefaultNumberTheoremWithin}
     \numberlike{Claim}{Lemma}
     \labelformat{Claim}{Claim~#1}
  \newtheorem{Theorem}{Theorem}
     \numberwithin{Theorem}{\DefaultNumberTheoremWithin}
     \numberlike{Theorem}{Lemma}
     \labelformat{Theorem}{Theorem~#1}
  \newtheorem{Corollary}{Corollary}
     \numberwithin{Corollary}{\DefaultNumberTheoremWithin}
     \numberlike{Corollary}{Lemma}
     \labelformat{Corollary}{Corollary~#1}
  \newtheorem{Proposition}{Proposition}
     \numberwithin{Proposition}{\DefaultNumberTheoremWithin}
     \numberlike{Proposition}{Lemma}
     \labelformat{Proposition}{Proposition~#1}
  
     \numberwithin{Conjecture}{\DefaultNumberTheoremWithin}
     \numberlike{Conjecture}{Lemma}
     \labelformat{Conjecture}{Conjecture~#1}

  \theoremstyle{definition}
  
     \numberwithin{Definition}{\DefaultNumberTheoremWithin}
     \numberlike{Definition}{Lemma}
     \labelformat{Definition}{Definition~#1}
  \theoremstyle{definition}
  
     \numberwithin{Question}{\DefaultNumberTheoremWithin}
     \numberlike{Question}{Lemma}
     \labelformat{Question}{Question~#1}
  \theoremstyle{definition}
  
     \numberwithin{Problem}{\DefaultNumberTheoremWithin}
     \numberlike{Problem}{Lemma}
     \labelformat{Problem}{Problem~#1}

  \theoremstyle{remark}
  
     \numberwithin{Remark}{\DefaultNumberTheoremWithin}
     \numberlike{Remark}{Lemma}
     \labelformat{Remark}{Remark~#1}
  \theoremstyle{remark}

     \numberwithin{Example}{\DefaultNumberTheoremWithin}
     \numberlike{Example}{Lemma}
     \labelformat{Example}{Example~#1}
  
     \labelformat{Case}{Case~#1}
     \numberwithin{Case}{Lemma}
  
     \labelformat{Step}{Step~#1}
     \numberwithin{Step}{Lemma}

  \let\epsilon\varepsilon
  \let\phi=\varphi
  \let\kappa=\varkappa
  \opn\dis{dis}
  \def\pnt{{\raise0.5mm\hbox{\large\bf.}}}
  
  \opn\Lex{Lex}

\begin{document}

  \title{Homology of powers of ideals: Artin--Rees numbers of syzygies and the Golod property}


  \author{J\"urgen Herzog}

  \address{Fachbereich Mathematik, Universit\"at Duisburg-Essen, Campus Essen, 45117
    Essen, Germany} \email{juergen.herzog@uni-essen.de}

  \author{Volkmar Welker}
  \address{Philipps-Universit\"at Marburg, Fachbereich Mathematik und Informatik,
            35032 Marburg, Germany} \email{welker@mathematik.uni-marburg.de}

  \author{Siamak Yassemi}
  \address{School of Mathematics, Statistics and Computer Science,
           College of Science, University of Tehran, Tehran, Iran, and School of
           Mathematics, Institute for Research in Fundamental Sciences (IPM), P.O. Box
           19395-5746, Tehran, Iran}
      \email{yassemi@ipm.ir}

  \begin{abstract}
    For an ideal $I$ in a regular local ring $(R,\mm)$ with residue class field $K = R/\mm$ or a standard graded 
    $K$-algebra $R$ we show that for $k \gg 0$ 
    \begin{itemize}
      \item the Artin--Rees number 
            of the syzygy modules of $I^k$ as submodules of the free modules from a free resolution is constant,
            and thereby present the Artin-Rees number as a proper replacement of regularity in the local situation,
      \item the ring $R/I^k$ is Golod, its Poincer{\'e}-Betti series is rational and the Betti numbers of the free
            resolution of $K$ over $R/I^k$ are polynomials in $k$ of a specific degree.
    \end{itemize}
    The first result is an extension of work of Kodiyalam and Cutkosky, Herzog \& Trung on the regularity of 
    $I^k$ for $k \gg 0$ from the graded situation to the local situation.
    The polynomiality consequence of the second result is an analog of work by Kodiyalam on the 
    behavior of Betti numbers of the minimal free resolution of $R/I^k$ over $R$.
  \end{abstract}

  \maketitle

\section{Introduction}
  Over the last 20 years the study of algebraic, homological and combinatorial properties of powers of 
  ideals has been one of the major topics in Commutative Algebra. In this paper we extend this in two 
  so far unexplored directions.

  \noindent {\sf Artin-Rees numbers :}
          The most important invariants of a graded ideal $I$ in a polynomial ring provided by the 
          Betti diagram are the projective dimension and the regularity. A result by Brodmann \cite{Brodmann} 
          shows that $\depth R/I^k$ and hence $\projdim I^k$ are constant for $k\gg 0$.
          It was shown in \cite{Kodiyalam2} and \cite{CutkoskyHerzogTrung} that the regularity 
          $\reg I^k$ of $I^k$ is a linear function for $k\gg 0$ (see also \cite{EisenbudHarris} and 
          \cite{EisenbudUlrich}) for structural results on the point of stabilization and constant 
          term of the linear function). While the regularity can only be defined for graded ideals, 
          the projective dimension is defined and is a finite number for any ideal in a regular local ring.
          Thus it is natural to ask: Which numerical invariant of an ideal in a regular local ring corresponds 
          to the regularity of a graded ideal in a polynomial ring? Does one obtain stability in the sense 
          of \cite{Kodiyalam2} and \cite{CutkoskyHerzogTrung} for high powers of $I$ ? We approach this question 
          by observing that for a graded ideal, the linearity of $\reg I^k$ for $k \gg 0$
          implies that there is an upper bound independent of $k$ for the degrees of the entries of the matrices 
          describing the syzygies of $I^k$. Thus in order to find a suitable replacement for the regularity 
          in the local case we need a measure that bounds the ``size'' of the entries of the syzygies. 
          We show that the Artin--Rees number (see \eqref{artinrees})
          of syzygy modules of $I^k$ inside the corresponding free module is constant for large $k$.
          We further support the choice of the Artin--Rees number as a
          substitute for regularity by comparing regularity and Artin--Rees number in \ref{comparison} for 
          finitely generated graded modules over polynomial rings which implies a linear upper bound on the regularity of $I^k$.

  \noindent {\sf Golod property :}
          In \cite{Kodiyalam1} Kodiyalam showed
          that the total Betti numbers $\beta_i^R(I^k)$ of an ideal $I$ in a Noetherian local ring 
          are polynomial functions for large $k$. In the case that $R$ is a polynomial
          ring and $I$ is a graded ideal, a certain refinement of this statement can be found 
          in the more recent paper \cite{HerzogWelker}.
          For an ideal $I$ in a regular local ring $(R,\mm)$ with $R/\mm = K$ or a polynomial 
          ring $R$ over $K$ we study the Betti numbers $\beta_i^{R/I^k}(K)$ of the
          free resolution of $K$ over $R/I^k$. We show in \ref{golod} that $R/I^k$ is Golod 
          for $k \gg 0$, a homological property that by definition implies
          trivial multiplication in the Tor algebra $\Tor_\bullet^R(R/I,K)$,  rationality of the 
          Poincar{\'e}-Betti series of $R/I^k$ by a result of Golod and thereby connects the Betti
          numbers $\beta_i^{R/I^k}(K)$ and $\beta_i^R(R/I^k)$. As a consequence we show in \ref{lem:bettipol} 
          that the Betti numbers $\beta_i^{R/I^k}(K)$ and the deviations $\epsilon_i(R/I^k)$
          are polynomial functions of a specific degree.
          We were inspired by a result of G.\ Levin \cite[Thm. 3.15]{Levin} saying that for any Noetherian local ring $(R,\mm)$ the  
          canonical epimorphism $R\to R/\mm^k$ is a Golod homomorphism for all $k\gg 0$.

  For the convenience of the reader and due to the lack of suitable references we recall in 
  Section \ref{sec:prelim} some basic facts about Artin--Rees numbers, and show in \ref{minmax} 
  how the degrees of the  entries of a matrix describing a graded submodule  $N$ of a free module 
  $F$ are bounded by the Artin--Rees number  $\rho(N,F)$. In Section~\ref{sec:asym} we consider the 
  asymptotic behaviour of the Artin--Rees numbers of the syzygies of powers of ideals, 
  prove \ref{growth} and \ref{comparison}.  In Section \ref{sec:golod}
  we study the Golod property of $R/I^k$ and prove \ref{golod} and \ref{lem:bettipol}.

  For all unexplained concepts from commutative algebra we refer the reader to the book \cite{Eisenbud}.

\section{Preliminaries regarding  Artin--Rees numbers}
  \label{sec:prelim}
  Let $(R,\mm)$ denote a Noetherian local ring or standard graded $K$-algebra with graded maximal ideal $\mm$,
  $M$ a finitely generated $R$-module and $N\subset M$ a submodule of $M$. In the graded case we assume that $M$
  is graded and $N$ is a graded submodule of $M$.

  By the Artin--Rees Lemma \cite[Lem. 5.1]{Eisenbud}, there exists an integer $r$ such that
  \begin{eqnarray}
     \label{artinrees}
     N\sect \mm^iM & = & \mm^{i-r}(N\sect \mm^{r}M) \quad \quad \text{for all~} i\geq r.
  \end{eqnarray}
  The smallest such number $r$ is called the {\em Artin-Rees number} and will be denoted $\rho(N,M)$.

  We denote by $S$ the associated graded ring $\gr_\mm(R)$ of $R$, and by $\nn$ the graded maximal ideal of $S$. Of course, if $R$ is standard graded,
  then $R=S$ and $\mm=\nn$. For an $R$-module $M$ and $x\in M$ with $x\neq 0$ we set $\ell(x)=x+\mm^{k+1}M$ where $k$ is the largest integer
  such that $x\in \mm^kM$. The element $\ell(x)\in \gr_\mm(M)$ is called the leading form of $x$.

  \begin{Proposition}
    \label{star}
    Let $N^*$ be the kernel of the  natural epimorphism $\gr_\mm(M)\to \gr_\mm(M/N)$ of graded $S$-modules.
    Then

    \begin{itemize}
       \item[(a)] $N^*=\sum_{x\in N} S\ell(x)$,
         \[
           (N^*)_k=(N\sect \mm^kM)/(N\sect \mm^{k+1}M)\quad \text{for all} \quad k,
         \]
         and
         \[
           \rho(N,M)=\max\{k\: (N^*/\nn N^*)_k\neq 0\}.
         \]
         In addition, in the graded case  $N^*=\sum_{x\in \homogen(N)} S\ell(x).$

      \item[(b)] for  $x_1,\ldots,x_r\in N$ such that $\ell(x_1),\ldots,\ell(x_r)$ generate $N^*$,  the elements $x_1,\ldots,x_r$ generate $N$.
    \end{itemize}
  \end{Proposition}

  \begin{proof}
    \begin{itemize}
      \item[(a)] It is clear that $\ell(x)\in N^*$ for all $x\in N$. Conversely, suppose that  $\ell(x)=x+\mm^{k+1}M$ belongs to $N^*$.
        Then $\project(x)+\mm^{k+1}W=0$, where $W=M/N$ and $\project\:\; M\to W$ is the canonical epimorphism. Hence there exists
        $y\in \mm^{k+1}W$ such that $\project(x)=y$. Let $z\in \mm^{k+1}M$ with $\project(z)=y$ and set $x'=x-z$. Then $x'\in N$ and $\ell(x')=\ell(x)$.

        Next observe that the exact sequence
        \[
          0\to N\to M\to W\to 0
        \]
        induces for all $k$ the exact sequence
        \[
          0\to N\sect \mm^kM\to \mm^kM\to \mm^kW\to 0.
        \]
        This shows that $(N^*)_k=(N\sect \mm^kM)/(N\sect \mm^{k+1}M)$ for all $k$, and this implies immediately that $\rho(N,M)=\max\{k\: (N^*/\nn N^*)_k\neq 0\}$.

      \item[(b)] Let $U=\sum_{i=1}^rRx_i$. We want to show that $U=N$. Let $x\in N$ with $\ell(x)=x+\mm^{k+1}M$. By assumption there exist $a_i\in R$ such that
        $\ell(x)=\sum_i\ell(a_i)\ell(x_i)$. It follows that $x'=x-y\in \mm^{k+1}M$, where $y=\sum_ia_ix_i\in U$. Let $\deg \ell(x')=t$. Then
        $t\geq  k+1$, and as before we find $z\in U$ such that $x''=x'-z\in \mm^{t+1}M$. In other words, $x-w\in \mm^{t+1}M$ where $w=y+z\in U$. Proceeding in this
        way we can find for any given number $s$ an element $u\in U$ such that $x-u\in \mm^{s+1}M$. Choosing $s=\rho(N, M)$, we see that for any $x\in N$ there exists
        $u\in U$ such that $x-u\in\mm N$. Thus we have shown that $N=U+\mm N$. Nakayama's lemma implies that $U=N$, as desired.
    \end{itemize}
  \end{proof}

  \begin{Corollary}
    \label{minmax}
    Suppose $R$ is standard graded and  $F$ is a finitely generated graded free $R$-module with homogeneous basis $e_1,\ldots,e_s$. Let  $N\subset F$ be a graded submodule
    of $F$. Then for any   minimal set of homogeneous generators of $x_1,\ldots,x_t$ of $N$ with
    \[
      x_i=\sum_{j=1}^sa_{ij}e_j, \quad i=1,\ldots,t,
    \]
    such that
    \begin{eqnarray}
      \label{min}
      \rho(N,F)\geq \min_{i=1,\ldots,t}\{\min_{j=1,\ldots,s}\{\deg a_{ij}\}\},
    \end{eqnarray}
    and
    \begin{eqnarray}
    \label{max}
    \rho(N,F)\geq \max_{i=1,\ldots,t}\{\min_{j=1,\ldots,s}\{\deg a_{ij}\}\}
    \end{eqnarray}
    for a suitable choice of $x_1,\ldots,x_t$.

    Here we use the convention that $\deg a=\infty$ if $a=0$.
  \end{Corollary}

  \begin{proof} We first prove \eqref{max}.
    By \ref{star}(a) there exist homogeneous elements $x_1,\ldots,x_r\in N$ such that the set of leading forms $$\ell(x_1),\ldots,\ell(x_r)$$ is a minimal set of
    generators of $N^*$. By \ref{star}(b), $x_1,\ldots,x_r$ is a system of generators of $N$. Since $R$ is standard graded, a suitable subset of $x_1,\ldots,x_r$,
    say, $x_1,\ldots,x_t$ is a minimal set of generators of $N$. Therefore, $\ell(x_1),\ldots,\ell(x_t)$ is part of a minimal
    set of generators of $N^*$. Hence, if  $\ell(x_i)=x_i+\mm^{k_i+1}F$ for $i=1,\ldots,t$, then it  follows from
    \ref{star}(a) that $k_i\leq \rho(N,F)$ for $i=1,\ldots,t$. Now if $x_i=\sum_{j=1}^sa_{ij}e_j$, then $k_i=\min\{\deg a_{ij}\:\; j=1,\ldots,s\}$. Thus the assertion follows.

    Inequality \eqref{max} is stronger than Inequality \eqref{min}. Moreover, the right hand side of Inequality \eqref{min} is independent of the chosen minimal
    system of generators. From this Inequality \eqref{min} follows for all minimal systems of generators of $N$.
  \end{proof}

\section{Asymptotic behavior of Artin-Rees numbers}
  \label{sec:asym}

  In this section we are going to consider the behavior of Artin-Rees numbers of the syzygy modules of powers of ideals and compare them in the graded case
  with their regularity. For a graded $R$-module $M$ over a polynomial ring $R$ of projective dimension $p$ we denote by
  $\reg_j := \max \{ k-j~:~\beta_{jk}^R (M) \neq 0\}$, $0 \leq j \leq p$, the regularity of
  its $j$\textsuperscript{th} syzygy module, here $\beta_{jk}^R(M)$ are the graded Betti numbers of $M$ over $R$.
  In particular, $\reg_0(M)$ is the maximal degree of a
  homogeneous generator of $M$. We write $\reg(M) := \max_{0 \leq j \leq p} \reg_j(M)$ for the regularity of $M$.
  For the sake of later use we also introduce here for an module $M$ over a standard graded or local ring $R$ the
  notation $\beta_i^R(M)$ for the $i$\textsuperscript{th} Betti number of $M$ over $R$.

  Since powers of ideals are best studied as the graded components of Rees algebras, which for graded ideals have a natural bigraded structure, one is led
  to consider bigraded modules.

  \begin{Lemma}
    \label{bigraded}
      Let $K$ be a field, $T=K[x_1,\ldots,x_n,y_1,\ldots,y_m]$ the standard  bigraded polynomial ring over $K$, $A=K[x_1,\ldots,x_n]$ and $W$ a
      finitely generated bigraded $T$-module. For each integer $k$ consider the finitely generated graded $A$-module
      \[
         W_k=\Dirsum_{j}W_{jk}.
      \]
      Then  $\reg_0(W_k)$ is constant for $k\gg 0$.
  \end{Lemma}

  \begin{proof}
    Let $w_1,\ldots,w_r$ be a set of homogeneous generators of $W$ with $\deg w_i=(a_i,b_i)$. We may assume that $a_1\leq a_2\leq \cdots\leq a_r$.
    Let $B=K[y_1,\ldots,y_m]$ be the polynomial ring over $K$ in the variables $y_1,\ldots,y_m$. Then
    \[
      W_k=\sum_{i=1}^rAB_{k-b_i}w_i
    \]
    for all $k$. It follows that $\reg_0(W_k)\leq a_r$.

    Assume that $B_{k-b_r}w_r\not\subset \sum_{i=1}^{r-1}AB_{k-b_i}w_i$ for all $k\geq b$, where $b=\max\{b_i\:\; i=1,\ldots,r\}$. Then $\reg_0(W_k)=a_r$ for all $k\geq b$.

    On the other hand, if $B_{k_1-b_r}w_r\subset \sum_{i=1}^{r-1}AB_{k_1-b_i}w_i$ for some $k_1\geq b$, then
    $B_{k-b_r}w_r\subset \sum_{i=1}^{r-1}AB_{k-b_i}w_i$ for all $k\geq k_1$. It follows that $W_k=W'_k$ for all $k\geq k_1$, where
    $W'=\sum_{i=1}^{r-1}Tw_i$. Applying induction on the number of generators of $W$, the desired conclusion follows.
  \end{proof}

  Let $S=R[y_1,\ldots,y_m]$ be the polynomial ring over $(R,\mm)$ in the variables $y_1,\ldots,y_m$, where according to our general assumption
  $R$ is either a Noetherian local ring or a standard graded $K$-algebra. We consider $S$ to be a graded ring by setting $\deg a=0$ for all
  $a\in R\setminus\{0\}$ and $\deg y_j=1$ for $j=1,\ldots,m$. In particular, for a finitely generated graded $S$-module $M$ each graded component $M_k$ is a finitely generated $R$-module.

  \begin{Proposition}
    \label{constant}
    Let $N\subset M$ be  graded $S$-modules. Then $\rho(N_k,M_k)$ is constant for $k\gg 0$.
  \end{Proposition}

  \begin{proof}
    By \ref{star} we have $\rho(N_k,M_k)=\reg_0(N_k)^*$. We let $W=M/N$ and $N^*$ be the kernel of the canonical epimorphism $\gr_\mm(M)\to \gr_\mm(W)$.
    The ring $\gr_\mm(S)=\gr_\mm(R)[y_1,\ldots,y_m]$ is naturally bigraded  and $N^*$ is a bigraded $\gr_\mm(S)$-module with
    \[
      (N^*)_{jk}= \Ker (\mm^{j}M_k/\mm^{j+1}M_k\To \mm^{j}W_k/\mm^{j+1}W_k).
    \]
    We set $(N^*)_k=\Dirsum_j(N^*)_{jk}$. Then we see that $(N^*)_k=(N_k)^*$ for all $k$, and hence  $\rho(N_k,M_k)=\reg_0(N^*)_k$ for all $k$.

    If $\embdim R=n$, then there exists an epimorphism $T=K[x_1,\ldots,x_n,y_1,\ldots,y_m]\to \gr_\mm(S)$ of standard bigraded $K$-algebras. Thus we may
    view $N^*$ to be a bigraded $T$-module. Applying \ref{bigraded}, we conclude that    $\rho(N_k,M_k)$ is constant for $k\gg 0$.
  \end{proof}

  Next we will apply this result to study the Artin-Rees numbers of the syzygies of powers of an ideal.
  Let $(\FF,\varphi)$ be a graded  free $R$-resolution of $M$. The $j$\textsuperscript{th} syzygy module of $M$ with respect to $\FF$ is defined to be the module
  $N_j=\Im(\varphi_j)\subset F_{j-1}$. In the case that $(\FF,\varphi)$ is  a graded minimal  free $R$-resolution of $M$, we set
  $\rho_j(M):=\rho(N_j, F_{j-1})$ for all $j\geq 1$ and $\rho_0(M) := \reg_0(M)$.

  \begin{Theorem}
    \label{growth}
    Let $(R,\mm)$ be a Noetherian local ring or a standard graded $K$-algebra with graded maximal ideal $\mm$, and $I\subset \mm$ an ideal. In the case
    that $R$ is graded, we assume that $I$ is a  graded ideal. Then for all $j\geq 1$ there exists an integer $c_j$ such that $\rho_j(I^k)=c_j$ for all $k\gg 0$.
    In particular, if $R$ is regular, then there exists an integer $c$ such that $\rho_j(I^k)\leq c$ for all $j\geq 1$ and all $k$.
  \end{Theorem}

  \begin{proof}
    The Rees ring $R(I)=\Dirsum_{k\geq 0}I^kt^k$  of $I=(f_1,\ldots,f_m)$ is a standard graded $R$-algebra with generators $f_it$ for $i=1,\ldots,m$.
    Let $S=R[y_1,\ldots,y_m]$ be the polynomial ring over $R$ in the indeterminates $y_1,\ldots,y_m$. We define the surjective $R$-algebra homomorphism
    $\project\: S\to R(I)$ with $\project(y_i)=f_it$ for $i=1,\ldots,m$. Then $R(I)=S/J$, where $J=\Ker \project$. Let $(\FF,\varphi)$ be a graded minimal
    free $S$-resolution  of $R(I)$, minimal in the sense that $\varphi(\FF)\subset \nn\FF$, where $\nn=(\mm, y_1,\ldots,y_m)$.


    We denote by $\FF^{(k)}$ the $k$\textsuperscript{th} homogeneous component of the resolution $\FF$. Then $\FF^{(k)}$ is a (not necessarily minimal)
    free $R$-resolution of $I^k$. Let $N_j$ be the $j$\textsuperscript{th} syzygy module of $R(I)$ viewed as an $S$-module. Then the  $k$\textsuperscript{th}
    component $N_j^{(k)}$  of $N_j$ is the $j^{th}$  syzygy module of $I^k$ with respect to the resolution $\FF^{(k)}$.

    By \ref{constant}, there exists an integer $c_j$ such that  $\rho(N_j^{(k)},F_{j-1}^{(k)})=c_j$ for $k\gg 0$.
    Let $\GG^{(k)}$ be a minimal graded free $R$-resolution of $I^k$, and $W_j^{(k)}$ be the $j^{th}$ syzygy module of $I^k$ with respect to $\GG$. Then for each
    $k$ there exists a free $R$-module $H_{j-1}^{(k)}$ such that $F_{j-1}^{(k)}=G_{j-1}^{(k)}\dirsum H_{j-1}^{(k)}$ and $N_j^{(k)}=W_j^{(k)}\dirsum H_{j-1}^{(k)}$. We deduce   that
    $\rho(N_j^{(k)},F_{j-1}^{(k)})=\rho(W_j^{(k)},G_{j-1}^{(k)})$ for all $k$.
    This yields the desired conclusion.
  \end{proof}

  The following result provides a comparison of the numbers $\rho_j(M)$ with the regularity $\reg(M)$ of $M$ in the case that $M$ is a graded module over the polynomial ring.

  \begin{Theorem}
    \label{comparison}
    Let $R=K[x_1,\ldots,x_n]$ be the polynomial ring, and $M$ a finitely generated graded $R$-module of projective dimension $p$. Then
    \[
      \reg_j(M)\leq \sum_{k=0}^j\rho_k(M)-j\quad \text{for}\quad j=0,\ldots,p.
   \]
   In particular, $\reg(M)\leq  \max\{\sum_{k=0}^j\rho_k(M)-j\:\; j=0,\ldots,p\}$.
  \end{Theorem}

  \begin{proof}
    Within the proof we write $\beta_j$ for $\beta_j^R(M)$.
    Let $(\FF,\partial)$ be a graded minimal free $R$-resolution of $M$  with
    $$F_j=\Dirsum_{k=1}^{\beta_j}S(-b_{jk}) \quad\quad\quad \text{and} \quad\quad\quad \alpha= (a_{rs})_{r=1,\ldots,\beta_j \atop s=1,\ldots, \beta_{j-1}}$$
    the matrix representing the differential
    $\partial_j\:\; F_j\to F_{j-1}$ with respect to homogeneous bases of $F_{j-1}$ and $F_j$.  The assertion of the theorem will follow once we have shown that
    \[
      \sum_{k=0}^j\rho_j(M)\geq  \max\{b_{jk}\:\; k=1,\ldots,\beta_j\}.
    \]
    This inequality will be shown by induction on $j$. For $j=0$, the inequality is an equality by the definition of $\rho_0(M)$. Now assume that $j>0$.
    By a suitable choice of a basis of $F_j$ we may assume that the matrix $\alpha$ satisfies
    $$\rho_j(M)\geq \max_{r=1,\ldots,\beta_j}\{\min_{s=1,\ldots,\beta_{j-1}}\{\deg a_{rs}\}\},$$
    see \ref{minmax}.

    Observe that $\deg a_{rs}=b_{jr}-b_{j-1,s}$ if $a_{rs}\neq 0$ and $\deg a_{rs}=\infty$, otherwise. Let $r_0$ be the index with the property that
    $b_{jr_0}=\max\{b_{jr}\:\; r=0,\ldots,\beta_j\}$,   $s_0$ the index with the property that $b_{j-1,s_0}=\max\{b_{j-1,s}\:\; s=0,\ldots,\beta_{j-1}\}$ and $s_1$
    the index with $\deg a_{r_0s_1}= \min\{\deg a_{r_0s}\:\; s=1,\ldots,\beta_{j-1}\}$. Then, by using induction on $j$, we obtain
    \[
      \rho_j(M)\geq \deg a_{r_0s_1}=b_{j r_0}-b_{j-1,s_1}\geq b_{j r_0}-b_{j-1,s_0}\geq b_{j r_0}-\sum_{k=0}^{j-1}\rho_k(M),
    \]
    from which it follows that  $\sum_{k=0}^{j}\rho_k(M)\geq b_{j r_0}$, as desired.
  \end{proof}

  By \cite{Kodiyalam1} and \cite{CutkoskyHerzogTrung} it is known that for a graded ideal $I\subset K[x_1,\ldots,x_n]$  there exist integers $a$ and $b$ such that
  $\reg(I^k)= ak+b$ for all $k$. By using \ref{comparison} and \ref{growth} we obtain

  \begin{Corollary}
    \label{known}
    Let $I\subset K[x_1,\ldots,x_n]$  be a graded ideal, then there exist integers $a$ and $b$ such that
    $\reg(I^k)\leq ak+b$ for all $k$.
  \end{Corollary}

\section{The Golod property}
  \label{sec:golod}

  Let $(R,\mm)$ be a Noetherian ring with residue class field $K$, or a standard graded $K$-algebra with graded maximal ideal $\mm$, and let $\xb=x_1,\ldots,x_n$ a minimal (homogeneous) system of generators of
  $\mm$.  We denote by $(K(R), \partial)$ the Koszul complex of $R$ with respect to $\xb$. Let  $Z(R)$, $B(R)$ and $H(R)$ denote the module of cycles, boundaries and the homology of $K(R)$.

  Recall (see \cite[Def. 5.5 and 5.6]{AvramovKustinMiller})
  that $R$ is said to be {\em Golod}, if for each subset $\mathcal{S}$ of
  homogeneous elements of $\Dirsum_{i=1}^nH_i(R)$  there exists a function $\gamma$, which is defined on the set of finite
  sequences of elements from $\Sc$  with values in $\mm\dirsum\Dirsum_{i=1}^nK_i(R)$,  subject to the following conditions:
  \begin{enumerate}
    \item[(G1)] if $h\in \Sc$, then $\gamma(h)\in Z(R)$ and $h=[\gamma(h)]$;
    \item[(G2)] if  $h_1,\ldots,h_m$ is a sequence in $\mathcal{S}$ with $m>1$,  then
    \[
      \partial\gamma(h_1,\ldots,h_m)=\sum_{\ell=1}^{m-1}\overline{\gamma(h_1,\ldots,h_\ell)}\gamma(h_{\ell+1},\ldots,h_m),
    \]
    where $\bar{a} = (-1)^{i+1}a$ for $a\in K_i(R)$.
  \end{enumerate}

  A function $\gamma$ defined on finite sequences from $\Sc$ with the properties (G1) and (G2) is called a {\em Massey operation} on $\mathcal{S}$

  In this section we want to prove the following

  \begin{Theorem}
    \label{golod}
    Let $(R,\mm)$ be a regular local ring or the polynomial ring over $K$  with graded maximal ideal $\mm$, and $I\subset \mm$ an ideal.
    We assume that $I$ is graded if $R$ is the polynomial ring.  Then the ring $R/I^k$ is Golod for $k\gg 0$.
  \end{Theorem}

  Given integers $i,s\geq 1$  and an integer $j\geq 0$ there is a natural $R$-module homomorphism
  $$
   \ast \: \left\{ \begin{array}{ccc}  I^j\tensor K_i(R/I^s) & \to & K_i(R/I^{s+j}) \\
                                                   a \tensor c & \mapsto & a \ast v
                   \end{array}
           \right.
  $$
  which is defined as follows: observing that $K_i(R/I^t)=K_i(R)/I^tK_i(R)$ for any integer $t\geq 1$, we let $K_i(R)\to K_i(R/I^t)$ be the
  canonical epimorphism which assigns to $u\in K_i(R)$ the residue class $u+I^tK_i(R)$. Now if  $a\in I^j$ and $v=u+I^sK_i(R)\in K_i(R/I^s)$. Then we
  set  $a\ast v :=au+I^{s+j}K_i(R)$. Obviously this map is well defined, that is, independent of the choice of $u$.
  By a simple calculation it follows that for $a\in I^j$, $b \in I^\ell$ and $v \in K_i(R/I^s)$
  \begin{eqnarray}
     \label{eq:ast}
     (ab)\ast v & = & a \ast ( b \ast v) \in K_i(R/I^{s+j+\ell})
  \end{eqnarray}
  Notice that if $v\in Z_i(R/I^s)$, then $a\ast v \in Z_i(R/I^{s+j})$. Indeed,  we have $v=u+I^sK_i(R)\in Z_i(R/I^s)$ if and only if  $\partial(u)\subset I^sK_{i-1}(R)$.
  In that case $\partial(au)\in I^{s+j}K_{i-1}(R)$,  and hence $\partial(a\ast v) =\partial(au)+I^{j+s}K_i(R) = 0$. Similarly one shows that
  $a\ast v \in B_i(R/I^{s+j})$, if $v\in B_i(R/I^s)$. Thus $\ast$ induces a map
  \[
    I^j\tensor H_i(R/I^s)\to H_i(R/I^{s+j}),\quad a\tensor [z] \mapsto a \ast [z] := [a \ast z].
  \]
  We will denote the image of the map $\ast : I^j\tensor H_i(R/I^s)\to H_i(R/I^{s+j})$ by $I^j \ast  H_i(R/I^s)\subset  H_i(R/I^{s+j})$.

  \medskip
  For the proof of the theorem  we shall need the following

  \begin{Lemma}
    \label{finite}
    There exists an integer $s'$ such that for $s \geq s'$
    $$I^j \ast H_i(R/I^s)=H_i(R/I^{s+j})$$
    for all integers $i\geq 1$ and $j\geq 0$.
  \end{Lemma}

  \begin{proof}
    Observe that $H_i(R/I^k)\iso H_{i-1}(I^k)$. By using this isomorphism it follows that $I^j\tensor H_i(R/I^s) \to H_i(R/I^{s+j})$ is surjective if and only
    if the map
    \begin{eqnarray}
      \label{eq:map}
      I^j\tensor H_{i-1}(I^s) & \to & H_{i-1}(I^{s+j})
    \end{eqnarray}
    with $a\tensor [z]\mapsto [az]$ is surjective. Thus it amounts to show that there exists an integer $s$ such that
    $I^j\tensor H_{i}(I^s)\to H_{i}(I^{s+j})$ is surjective for all $i=0,\ldots,n-1$ and all $j\geq 0$.

    To this end we consider the Koszul complex $K(\xb;R(I))$ of the sequence $\xb$ with values in the Rees ring $R(I)$ of $I$. Recall that
    $H_i(\xb;R(I))$ is a finitely generated graded $R(I)$-module with graded pieces
    \begin{eqnarray}
      \label{eq:griso}
      H_i(\xb;R(I))_k & = & H_i(I^k) , \quad\text{for all} \quad k,
    \end{eqnarray}
    where the graded $R(I)$-module structure is given by \eqref{eq:map}.
    Thus for each $i$ there exists an integer $s_i$ such that
    \[
      H_i(I^{s_i}) = I^{k-s_i}H_i(\xb;R(I))_{s_i} = H_i(\xb;R(I)_{k} = H_i(I^k)
    \]
    for all $k\geq s_i$.

    Hence
    the number $s'=\max\{s_0,\ldots,s_{n-1}\}$ satisfies the condition of the lemma.
  \end{proof}

  \begin{proof}[Proof of \ref{golod}]
    In the case that $R$ is the polynomial ring, we denote by $\hat{R}$ the $\mm$-adic completion of $R$.
    Since $H(\hat{R}/I\hat{R})=H(R/I)$, we may replace $R$ by its completion and may therefore assume  that $R$ is local.

    We claim the following: for any integer  $r\geq 1$, there exists an integer $s_r$ such that for all $k\geq s_r$ and each
    homogeneous subset $\Sc\subset \Dirsum_{i=1}^nH_i(R/I^k)$ there exists a function $\gamma$,
    defined on  the set of sequences of elements of $\Sc$ of length $\leq r$, such that (G1) and (G2) hold.

    The claim will yield the desired result, since  for any such function we have that
    $\gamma(h_1,\ldots,h_r)\in K_a(R/I^k)$ where $a\geq 2r-1$. Hence  if $r\geq n/2+1$, we necessarily have $\gamma(h_1,\ldots,h_r)=0$.

    We will prove the claim by induction on $r$. For $r=1$, we may choose $s_r=1$ and the assertion is trivial.

    Now let $r\geq 1$ and assume that the claim is proved for all integers $\leq r$. By \ref{finite}
    there exists an integer $s$ such that $I^{k-s}H_i(R/I^s)=H_i(R/I^{k})$
    for all $k\geq s$ and all $i>0$. We set $s_{r+1}=\max\{s_r,s\}+1$. Let  $\mathcal{G}=\{g_1,\ldots,g_t\}$ be a homogeneous $K$-basis of
    $\Dirsum_{i=1}^nH_i(R/I^{s_r})$. By induction hypothesis there exists a function $\gamma$,
    defined on  the set of sequences of elements of $\mathcal{G}$ of length $\leq r$, such that (G1) and (G2) hold.

    Let $k\geq s_{r+1}$ and let $\Sc\subset \Dirsum_{i=1}^nH_i(R/I^k)$ be any  set of homogeneous elements.
    We are going  to define a function $\gamma$ on sequences from $\Sc$ of length $r+1$ satisfying (G1) and (G2).

    First, let $h_1,\ldots, h_m$ be any sequence of elements $h_i\in \Sc$ of length $m\leq r$. Since $s_{r+1}>s$, each $h_i$ can be written as
    $h_i=\sum_{j=1}^t a_{ij}g_j$ with $a_{ij} \in I^{k-s_r}$ and all $g_j\in H_t(R/I^{s_r})$ if $h_i\in H_t(R/I^k)$.
    We define $\gamma$ on this sequence by multi-linear extension, that is, we set
    \[
      \gamma(h_1,\ldots,h_m) := \sum_{j_1=1}^t\sum_{j_2=1}^t\cdots \sum_{j_m=1}^t
                                        (a_{1j_1}a_{2j_2}\cdots a_{mj_m}) \ast \gamma(g_{j_1},\ldots,g_{j_m}).
    \]
    Here we consider $a_{1j_1}a_{2j_2}\cdots a_{mj_m}$ as an element of $I^{k-s_r}$ so that $\gamma(h_1,\ldots,h_m)\in K(R/I^k)$.

    Next we verify that the function $\gamma$ satisfies (G1) and (G2) on sequences from $\Sc$ of length $\leq r$:

    \begin{itemize}
    \item[(G1)]
      Let $h\in\Sc$ with presentation $h=\sum_{j=1}^t a_{j}g_j$ where  $a_j\in I^{k-s_r}$ for all $j$.
      Since $\gamma(g_j)\in Z(R/I^{s_r})$ it follows that $a_{j} \ast \gamma(g_j)\in Z(R/I^{k})$, and hence we see
      that $\gamma(h)=\sum_{j=1}^t a_{j} \ast \gamma(g_j)$ belongs to $Z(R/I^{k})$ as well. Moreover, we have
      \[
        [\gamma(h)]=\sum_{j=1}^t [a_{j} \ast \gamma(g_j)]= \sum_{j=1}^t a_{j} \ast [\gamma(g_j)] = \sum_{j=1}^t a_{j}g_j=h.
      \]

    \item[(G2)]
      Since $\partial$ is linear with respect to $\ast$ we have
      \[
        \partial \gamma(h_1,\ldots,h_m)= \sum (a_{1j_1}a_{2j_2}\cdots a_{mj_m}) \ast \partial \gamma(g_{j_1},\ldots,g_{j_m}),
      \]
      where the sum is taken  over all $j_k$  ranging  between 1 and $t$. For each of the summands in $\partial \gamma(h_1,\ldots,h_m)$, by using
      $a \ast (v+w) = a\ast v + a \ast w$ and $(ab) \ast (vw) = (a \ast v) (b \ast w)$, we have
      \begin{eqnarray*}
        & & (a_{1j_1}a_{2j_2}\cdots a_{mj_m}) \ast \partial \gamma(g_{j_1},\ldots,g_{j_m})\\
        &=& (a_{1j_1}a_{2j_2}\cdots a_{mj_m}) \ast  \sum_{\ell=1}^{m-1} \overline{\gamma(g_{j_1},\ldots,g_{j_\ell})} \gamma(g_{j_{\ell+1}},\ldots,g_{j_m})\\
        &=& \sum_{\ell=1}^{m-1} (a_{1j_1}a_{2j_2}\cdots a_{mj_m}) \ast \Big( \overline{\gamma(g_{j_1},\ldots,g_{j_\ell})} \gamma(g_{j_{\ell+1}},\ldots,g_{j_m})\Big)\\
        &=& \sum_{\ell=1}^{m-1}\Big((a_{1j_1}a_{2j_2}\cdots a_{\ell j_\ell}) \ast \overline{\gamma(g_{j_1},\ldots,g_{j_\ell})}\Big) \Big((a_{\ell+1j_{\ell+1}}\cdots a_{mj_m}) \ast \gamma(g_{j_{\ell+1}},\ldots,g_{j_m})\Big).
      \end{eqnarray*}
      Summing over all indices $j_t$ independently yields the desired identity (G2).
    \end{itemize}

    In order to define $\gamma$ on sequences  of elements of $\Sc$ of length $r+1$, it suffices to show that for any sequence $h_1,\ldots,h_{r+1}$ of elements from $\Sc$,
    the element
    \[
      b = \sum_{\ell=1}^{r}\overline{\gamma(h_1,\ldots,h_\ell)}\gamma(h_{\ell+1},\ldots,h_{r+1}),
    \]
    is a boundary of $K(R/I^k)$. To this end, observe that $b$ is a linear combination of expressions of the form
    \[
      b'=\sum_{\ell=1}^{r}\overline{\gamma(g_{k_1},\ldots,g_{k_\ell})}\gamma(g_{k_{\ell+1}},\ldots,g_{k_{r+1}})
    \]
    with coefficients in $I^{(r+1)(k-s_r)}\subset I^{k-s_r}$. Let $a$ be the coefficient of $b'$. Since $b'$ is a
    boundary in $K(R/I^{s_r})$ and $a\in I^{k-s_r}$,  it follows that $ab'$  is a boundary in $K(R/I^{k})$. Thus $b\in B(R/I^k)$, as desired.
  \end{proof}

  In the following we consider for a Noetherian local ring $(R,\mm)$ with residue field $K = R/\mm$
  or a standard graded $K$-algebra $R$ with graded maximal ideal $\mm$
  its deviations $\epsilon_i(R)$ and the Betti numbers $\beta_i^R(K)$ of the minimal free resolution of
  $K=R/\mm$ over $R$.  We refer the reader to \cite{Avramov} and \cite{GulliksenLevin} for standard facts about these invariants.
  The characterization of the Golod property in terms of Poincar\'e-Betti series (see \cite[Cor. 4.2.4]{GulliksenLevin}
  or \cite[(5.0.1)]{Avramov} and \ref{golod}) immediately yield the following result.

  \begin{Corollary}
    \label{poincaregolod}
    Let $(R,\mm)$ be a regular local ring with $K = R/\mm$ or the polynomial ring over $K$ with graded maximal ideal $\mm$, and $I\subset \mm$ an ideal.
    We assume that $R$ is of dimension $d$ and that $I$ is graded if $R$ is the polynomial ring. Then for $k \gg 0$
    the multiplication on $Tor_*^R(R/I^k,K)$ is trivial and
    \begin{eqnarray*}
      \label{eq:golod}
      \sum_{i \geq 0} \beta_i^{R/I^k}(K) z^i & = & \frac{(1+z)^d}{1-z\sum_{i=1}^d \beta^R_i(R/I^k) z^i}
    \end{eqnarray*}
  \end{Corollary}

  The next result will exploit this Corollary in order to obtain specific information about the growth of
  $\beta_i^{R/I^k}(K)$ and $\epsilon_i(R/I^k)$ as a function of $k$.
  For the formulation of the result we denote for an ideal $I$ by $\ell(I)$ its analytic spread.

  \begin{Proposition}
    \label{lem:bettipol}
    Let $I$ be an ideal in the regular local ring $R$ of dimension $d$ or $I$ a graded ideal in
    $R = K[x_1,\ldots, x_d]$. If $\ell(I) \geq 2$ then for $i \geq 0$
    \begin{itemize}
      \item[(i)] $\beta^{R/I^k}_i(K)$ is
         a polynomial of degree $(\ell(I)-1 ) \cdot \lfloor \frac{i}{2}\rfloor$ in $k$ for $k \gg 0$.
      \item[(ii)] $\epsilon_i(R/I^k)$ is
         a polynomial of degree $(\ell(I)-1) \cdot \lfloor \frac{i+1}{2}\rfloor$ in $k$ for $k \gg 0$.
    \end{itemize}
  \end{Proposition}

  \begin{proof}[Proof of \ref{lem:bettipol} (i)]
    By \ref{golod} the ring $R/I^k$ is Golod for $k \gg 0$. Hence
    by \ref{poincaregolod} the Equation \eqref{eq:golod} is valid for $k \gg 0$.
    Multiplying with the denominator of the right hand side of \eqref{eq:golod}
    we obtain for $2 \leq i$ :
    \begin{eqnarray}
      \label{eq:infinite}
      \beta_i^{R/I^k} (K) & = & \sum_{l = 1}^{\min\{i-1,d\}} \beta_{i-1-l}^{R/I^k}(K) \beta^R_{l}(R/I^k) + {d \choose i}
    \end{eqnarray}
    Now by \cite[Cor. 7]{Kodiyalam1} each $\beta^R_i(R/I^k)$ is a polynomial in $k$ for $k \gg 0$.
    We assume that $k$ is large enough to satisfy the two preceding conclusions.
    We know that $\beta_0^{R/I^k} (K) = 1$ and $\beta_1^{R/I^k}(K) = d = \dim R$. In addition
    we know by \cite[Prop. 2.2]{HerzogWelker} that $\deg \beta_1^R(R/I^k) = \ell(I) -1 \geq \cdots \geq \deg \beta_n^{R}(R/I^k)$.
    Hence we deduce from \eqref{eq:infinite} and induction on $i$ that
    $\beta^R_i(K)$ is a polynomial of degree $(\ell(I)-1) \cdot \lfloor \frac{i}{2}\rfloor$ in $k$.
  \end{proof}

  Before we can proceed to the proof of \ref{lem:bettipol} (ii) we need some simple calculations
  relating $\epsilon_i(R/I)$ and $\beta_i^{R/I}(K)$ for general regular local rings $(R,\mm)$
  or polynomial rings $R$ and (graded) ideals $I$.
  Recall that for $i \geq 0$ the deviations $\epsilon_i(R/I)$ can be defined through

  \begin{eqnarray}
    \label{eq:defining1}
    \label{eq:defining2}
       \prod_{i \geq 0} \frac{(1+z^{2i+1})^{\epsilon_i(R/I)}}{(1-z^{2i+2})^{\epsilon_{i+1}(R/I)}} & = & \prod_{i\geq 0} (1 + (-1)^i z^{i+1})^{(-1)^i \cdot \epsilon_{i}(R/I)}  =
       \sum_{i \geq 0} \beta_i^{R/I}(K) z^i
  \end{eqnarray}

  We now assume that $k \gg 0$ is large enough so that Equation \eqref{eq:golod} from \ref{poincaregolod} is valid
  for $R/I^k$. For sake of simple notation we will for a moment abbreviate $\epsilon_i(R/I^k)$ by $\epsilon_i$ and $\beta_i^{R}(R/I^k)$ by $\beta_i$.

  Taking the logarithm on the left hand side of \eqref{eq:defining1} we obtain:

  \begin{eqnarray}
    \sum_{i\geq 0} (-1)^i \epsilon_i \log (1+(-1)^i z^{i+1}) & = &
                  \sum_{i\geq 0} (-1)^i \epsilon_i \sum_{n\geq 1} \frac{(-1)^{n-1}}{n} (-1)^{ni} z^{n(i+1)} \label{eq:logdev} \\
                                                       & = &
                  \sum_{m\geq 1} \Big( \sum_{n|m} (-1)^{m/n} \frac{\epsilon_{m/n-1}}{n} \Big) (-z)^{m}
                  \nonumber 
  \end{eqnarray}

  Taking logarithm on the right hand side of \eqref{eq:golod}
  we obtain:
  \begin{eqnarray}
    \label{eq:logbetti}
    d \log (1 + z)- \log(1-\sum_{i=1}^d \beta_i z^{i+1}) & = &
                               \sum_{i\geq 1} \Big( \frac{(-1)^{i-1}d}{i} z^i + \frac{1}{i} \big(\sum_{j=1}^d \beta_j z^{j+1}\big)^i\Big)
  \end{eqnarray}

  By comparing coefficients of $z^{i+1}i$ from \eqref{eq:logdev} and \eqref{eq:logbetti} we obtain for $i \geq 0$
  \begin{eqnarray}
    \label{eq:deviation}
    (-1)^{i+1} \Big( \sum_{n|i+1} (-1)^{\frac{i+1}{n}} \frac{\epsilon_{(i+1)/n-1}}{n} \Big) & = &
                                                  \frac{(-1)^{i+1} d}{i+1} + \sum_{{j_1 + \cdots + j_n+n = i+1} \atop {1 \leq j_1,\ldots, j_n \leq d}}
                                                     \frac{1}{n} \beta_{j_1} \cdots \beta_{j_n}
  \end{eqnarray}

  \smallskip

  \begin{proof}[Proof of \ref{lem:bettipol} (ii)]
    As in the proof of \ref{lem:bettipol} (i) we assume $k$ is large enough so that $R/I^k$ is Golod and each
    $\beta^R_j(R/I^k)$ is a polynomial in $k$.
    Again we proceed by induction. For $i = 0,1$, we have $\epsilon_0(R/I^k) = d$ and
    $\epsilon_1(R/I^k) = \beta_1^R(R/I^k)$. The claim holds since $\beta_1^R(R/I^k)$
    is a polynomial of degree $\ell(I) -1$.

    Now let $i \geq 2$.
    By \eqref{eq:deviation} we can express $\epsilon_{i}(R/I^k)$ as a linear
    combination of $\epsilon_j (R/I^k)$ for $j < \lfloor \frac{i+1}{2} \rfloor -1$ and products of
    $\beta_{j_1}^{R}(R/I^k) \cdots \beta_{j_n}^{R} (R/I^k)$ for
    $1 \leq j_1,\ldots, j_n \leq d$ and $j_1+\cdots + j_n +n = i+1$. We first show that this
    second summand has the right degree as a polynomial in $k$.

    \noindent{\sf Claim:} The right hand side of \eqref{eq:deviation}
    is a polynomial of degree $(\ell(I)-1)\lfloor \frac{i+1}{2}\rfloor$ in $k$.

    $\triangleleft$
    For $1  \leq m \leq n$ the Betti number
    $\beta_{j_m}^R(R/I^k)$ is a polynomial of degree $\leq \ell(I) -1$ in $k$ by
    \cite[Prop. 2.2]{HerzogWelker}. It follows that $\beta_{j_1}^{R}(R/I^k) \cdots \beta_{j_n}^{R} (R/I^k)$ is a
    polynomial of degree $\leq (\ell(I)-1) \cdot n$ in $k$.
    Since $1 \leq j_1,\ldots, j_n$ the maximal degree is achieved for $n = \lfloor \frac{i}{2} \rfloor$.
    For $i$ even we can choose $j_1 = \cdots = j_n = 1$ and for $i$ odd we can
    choose $j_1 = \cdots = j_{n-1} = 1$ and $j_{n} = 2$. Since $\beta_1^R(R/I^k)$ is of
    degree $\ell(I)-1$ in $k$ it follows from $\ell(I) \geq 2$ and \cite[Prop 2.2, Rem. 2.5]{HerzogWelker} that
    $\beta_2^R(R/I^k)$ is also of degree $\ell(i)-1$ in $k$. Hence the asserted bound is achieved
    and the claim follows since the sum on the right hand side of \eqref{eq:deviation}
    runs over a positive linear combination of products of polynomials with positive leading coefficients.
    $\triangleright$

    By induction hypothesis for $j \leq \lfloor \frac{i+1}{2} \rfloor -1$ we have that $\epsilon_j(R/I^k)$ is a polynomial
    of degree
    $$(\ell(I)-1) \lfloor \frac{j+1}{2} \rfloor \leq (\ell(I)-1) \lfloor \frac{i/2}{2} \rfloor <  (\ell(I)-1) \lfloor \frac{i+1}{2} \rfloor.$$
    Hence there is no contribution in degree $(\ell(I)-1) \lfloor \frac{i+1}{2} \rfloor$
    from the $\epsilon_{j}(R/I^k)$ for $j < i$ on the right hand side of \eqref{eq:deviation}.
    Thus the assertion follows.
  \end{proof}

\section*{Acknowledgment}
  We thank Aldo Conca for suggesting the study of the asymptotic growth of deviations.
  Part of this work was carried out while the third author was visiting
  the Department of Mathematics of Universit\"at Duisburg-Essen. He is grateful for 
  its hospitality.

\end{document}